\documentclass[a4paper,twoside,12pt,leqno,
]{article}
\usepackage{amsmath,amsthm,amssymb,enumerate}
\usepackage{graphicx,  float, afterpage,array,multicol,eucal, url,color}

\swapnumbers
\theoremstyle{plain}
\newtheorem{theorem}{Theorem}[section]
\newtheorem{lemma}[theorem]{Lemma}
\newtheorem{corollary}[theorem]{Corollary}

\newtheorem{sw}[theorem]{The Stong-Whitehead theorem}
\newtheorem{cvl}[theorem]{Whitehead's cut-vertex lemma}
\newtheorem{stong}[theorem]{Stong's cut-vertex lemma}

\theoremstyle{definition}

\newtheorem{algorithm}[theorem]{Algorithm}

\newtheorem{definitions}[theorem]{Definitions}

\newtheorem{notation}[theorem]{Notation}
\newtheorem{history}[theorem]{History}

\newtheorem{content}[theorem]{Content}

\newtheorem{notes}[theorem]{Notes}

\newtheorem{hypotheses}[theorem]{Hypotheses}
\newtheorem{review}[theorem]{Review}
\numberwithin{equation}{theorem}

\DeclareMathOperator{\Aut}{Aut}

\DeclareMathOperator{\rel}{\,rel\,}
\DeclareMathOperator{\supp}{supp}
\DeclareMathOperator{\core}{core\null_\ast\mkern-2mu}
\DeclareMathOperator{\mcore}{modelcore\null_\ast\mkern-2mu}
\DeclareMathOperator{\Wh}{\textsc{W\hskip-1pt h}_\ast\mkern-2mu}
\DeclareMathOperator{\Who}{\textsc{W\hskip-1pt h}}
\DeclareMathOperator{\cuts}{\textsc{cuts}}
\DeclareMathOperator{\Cayley}{Cayley}
\DeclareMathOperator{\BST}{BassSerre}
\DeclareMathOperator{\Cl}{\textsc{Cl}}
\def \bigast {\operatornamewithlimits{\text{\LARGE $\ast$}}}
\def \bigcupp {\operatornamewithlimits{\text{\small $\bigcup$}}}

\def\d1{\discretionary{-}{}{-}}
\def\coloneq{\mathrel{\mathop\mathchar"303A}\mkern-1.2mu=}

\renewcommand{\phi}{\varphi}
\renewcommand{\le}{\leqslant}
\renewcommand{\ge}{\geqslant}

\begin{document}

\pagestyle{myheadings}
\markboth{Sandwiching between free-product factors}
{Warren Dicks}

\title{On free-group  algorithms that sandwich \\a  subgroup  between  free-product factors}

\author{
Warren Dicks\footnote{Partially supported by
Spain's Ministerio de Ciencia e Innovaci\'on
through Project MTM2011-25955.}}

\date{\small\today}

\maketitle

\noindent\textbf{Abstract.} Let $F$ be a finite-rank free group and   $H$ be a finite-rank
subgroup of~$F$.  We discuss proofs of two algorithms that sandwich $H$ between
an upper-layer  free-product factor  of $F$ that contains $H$
and a lower-layer  free-product factor  of $F$ that is contained in $H$.

Richard Stong showed that the unique   smallest-possible  upper layer, denoted
\mbox{$\Cl(H)$},
is visible in the output of
the polynomial-time cut-vertex algorithm of  J.\,H.\,C.\,White\-head.
Stong's proof used  bi-infinite paths in a Cayley tree and  sub-surfaces of
a three-manifold.
We give a variant of his proof that uses edge-cuts of the Cayley tree induced by edge-cuts of  a Bass-Serre tree.

A.\,Clifford and  R.\,Z.\,Gold\-stein gave an  exponential-time algorithm that
determines whether or not the trivial subgroup is the  only possible lower layer.
Their proof used Whitehead's three-manifold techniques. We give a variant of their proof that uses
 Whitehead's cut-vertex results, and thereby obtain a somewhat simpler   algorithm that
yields a lower layer of maximum-possible rank.

\medskip

{\footnotesize

 Cut-vertex algorithm.   Cut-vertex lemma.  Clifford-Gold\-stein algorithm.}

\bigskip

\section{Introduction}\label{sec:1}

\begin{definitions}
For any set $E$, we let
\mbox{$\langle\mkern1mu E\mkern2mu\vert\mkern12mu\rangle$}
denote  the free group on~$E$.
By a \textit{basis}
of \mbox{$\langle\mkern1mu E\mkern2mu\vert\mkern12mu\rangle$}, we mean
a  free-generating set of \mbox{$\langle\mkern1mu E\mkern2mu\vert\mkern12mu\rangle$}.
By a \textit{sub-basis} of \mbox{$\langle\mkern1mu E\mkern2mu\vert\mkern12mu\rangle$}
we mean a subset of a basis
of~\mbox{$\langle\mkern1mu E\mkern2mu\vert\mkern12mu\rangle$}. 
We let \mbox{$\Aut \langle\mkern1mu E\mkern2mu\vert\mkern12mu\rangle$} denote
the group of automorphisms of \mbox{$\langle\mkern1mu E\mkern2mu\vert\mkern12mu\rangle$} 
acting on the right as exponents.

 For any subset $Z$
of~\mbox{$\langle\mkern1mu E\mkern2mu\vert\mkern12mu\rangle$}, we let
 \mbox{$\langle Z \rangle$}  denote  the subgroup of~\mbox{$\langle\mkern1mu E\mkern2mu\vert\mkern12mu\rangle$}
generated by $Z$.  We let    \mbox{$\supp(Z \rel E)$} denote  the $\subseteq$-smallest subset of  $E$ such that
\mbox{$Z  \subseteq \langle \,\supp(Z \rel E)\, \rangle$}.   We let \mbox{$\Cl(Z)$} denote the intersection of all the
 free-product factors (generated by sub-bases)  of \mbox{$\langle\mkern1mu E\mkern2mu\vert\mkern12mu\rangle$}
that contain~$Z$.
\end{definitions}

\begin{hypotheses}\label{hyps:1} Throughout, let  $E$ be a finite set,
 let $Z$  be a finite subset
of~\mbox{$\langle\mkern1mu E\mkern2mu\vert\mkern12mu\rangle$}, and
 let \mbox{$E_Z$} denote  \mbox{$\supp(Z \rel E)$}.
\end{hypotheses}

\begin{history} Recall  Hypotheses~{\normalfont\ref{hyps:1}}.

$\bullet$
In~\cite[publ.\,1936]{W}, J.\,H.\,C.\,Whitehead gave his true-word
and cyclic-word  cut\d1ver\-tex algorithms, and the former determines whether or not
$Z$ is a sub-basis of~\mbox{$\langle\mkern1mu E\mkern2mu\vert\mkern12mu\rangle$}.
A little later, in~\cite[publ.\,1936]{W2}, he gave an
exponential-time,  gen\-eral-purpose algo\-rithm which has  largely overshadowed the
easier-to-prove, polynomial-time, limited-use algorithm.
We wish to emphasize that  the cut\d1ver\-tex  algorithm
suffices to efficiently sandwich  a subgroup between two free\d1product factors.

Whitehead defined a certain finite graph  which we   denote \mbox{$\Wh(Z \rel E_Z)$}.
He  observed that if some vertex of
\mbox{$\Wh(Z \rel E_Z)$} is what we  call a  Whitehead cut-vertex,
then it is straightforward to construct an automorphism  of
\mbox{$\langle\mkern1mu E\mkern2mu\vert\mkern12mu\rangle$}
 that  strictly  reduces  the  total  $E$-length of~$Z$.
Clearly, one then has an algorithm (with choices) which constructs some  
\mbox{$\Psi \in \Aut \langle\mkern1mu E\mkern2mu\vert\mkern12mu\rangle$} such that
 \mbox{$\Wh(Z^{\mkern2mu\Psi} \rel E_{Z^{\Psi}})$} has no Whitehead cut-vertices.
It then remains to extract information from  \mbox{$\Psi$} and \mbox{$Z^{\mkern2mu\Psi}$}.
For example, it will transpire that the rank of  \mbox{$\Cl(Z)$} is
\mbox{$\vert  E_{Z^{\Psi}} \vert$}.  One reason this is interesting
is that Edward\,\,C.\,Turner~\cite[Theorem~1]{Turner}  showed that 
 the rank of \mbox{$\Cl(Z)$} is  \mbox{$\vert E \vert$} if and only if
$Z$ is a \textit{test set} for injective endomorphisms  of~\mbox{$\langle\mkern1mu E\mkern2mu\vert\mkern12mu\rangle$}
to be  automorphisms, that is, each injective endomorphism 
of~\mbox{$\langle\mkern1mu E\mkern2mu\vert\mkern12mu\rangle$}
that maps \mbox{$\langle Z \rangle$} onto 
itself is an automorphism.  

Using a three-manifold model of \mbox{$\Wh(Z \rel E_Z)$},   Whitehead  proved  a
cut\d1ver\-tex lemma:
 If \mbox{$Z$} is a sub-basis of~\mbox{$\langle\mkern1mu E\mkern2mu\vert\mkern12mu\rangle$},
then \mbox{$Z^{\mkern2mu\Psi} \subseteq  E^{\pm 1}$}.

  Hence,
  \mbox{$Z$} is a sub-basis of~\mbox{$\langle\mkern1mu E\mkern2mu\vert\mkern12mu\rangle$}
if and only if \mbox{$Z {\,\cap\,} Z^{-1} \hskip-.5pt = \hskip-.5pt \emptyset$}
and  \mbox{$Z^{\mkern2mu\Psi} \subseteq  E^{\pm 1}$};
in this event,
 \mbox{$Z \cup (E^{\mkern1mu\Psi^{-1}} \mkern-6mu{-}  Z^{\pm 1})$} is a
basis of~\mbox{$\langle\mkern1mu E\mkern2mu\vert\mkern12mu\rangle$}.

Set  \mbox{$E\mkern1mu' \coloneq E^{\mkern1mu\Psi^{-1}}$} \hskip-4pt
and \mbox{$E\mkern1mu'\mkern-6mu_Z \coloneq \supp(Z \rel E\mkern1mu')$}.
Expressing the elements  of \mbox{$Z^{\mkern2mu\Psi}$}
in terms of \mbox{$E$}  is equivalent to expressing the elements of $Z$  in terms of  \mbox{$E\mkern1mu'$}.
The important point is  that
\mbox{$\Wh(Z \rel E\mkern1mu'\mkern-6mu_Z)$} is
isomorphic to \mbox{$\Wh(Z^{\mkern2mu\Psi} \rel E_{Z^{\Psi}})$}  and, hence,
has no Whitehead cut-vertices.

In \cite[publ.\,1997]{Stong}, Richard Stong used bi-infinite paths in a Cayley tree and
sub-surfaces  homologous to an essential  disk in  a three-manifold to  prove a more general 
  cut-vertex lemma: The set \mbox{$E\mkern1mu'\mkern-6mu_Z$} is a basis of \mbox{$\Cl(Z)$},
and, for each free-product   factorization \vspace{-2.5mm} \mbox{$\Cl(Z) = \bigast\limits_{i\in I} H_i$}  such   that
\mbox{$Z \subseteq \bigcupp\limits_{i\in I} H_i$}, the set \mbox{$E\mkern1mu'\mkern-6mu_Z$} 
 contains a basis of each~\mbox{$H_i$}. 

Not only can a basis of \mbox{$\Cl(Z)$}  be computed  efficiently, but also there are only 
finitely many possibilities for the sets
 \mbox{$\{H_i\}_{i \in I}$}, and they can all be computed efficiently.
To see how Stong's cut-vertex lemma generalizes Whitehead's, notice that
if \mbox{$Z$} is a sub-basis   of
\mbox{$\langle\mkern1mu E\mkern2mu\vert\mkern12mu\rangle$},  then
\mbox{$ \Cl(Z) = \langle Z \rangle = \bigast\limits_{z \in Z} \langle z \rangle$}
and \mbox{$Z \subseteq \bigcupp\limits_{z\in Z} \langle z \rangle$},  and, here,
for \mbox{$E\mkern1mu'\mkern-6mu_Z$}~to contain  a basis
of each~\mbox{$\langle z \rangle$},  which is necessarily \mbox{$\{z\}$} or  \mbox{$\{z^{-1}\}$},
one must   have \mbox{$(E\mkern1mu'\mkern-6mu_Z)^{\pm 1} \supseteq Z$}, and, hence,
\mbox{$ E^{\pm 1} \supseteq Z^{\mkern2mu\Psi}$}.

\smallskip

$\bullet$ In \cite[publ.\,2010]{CG0},  A.\,Clifford and  R.\,Z.\,Gold\-stein  revisited
Whitehead's three-manifold techniques
and constructed  an ingenious  exponential-time   algorithm  which determines whether or not
 some element of  \mbox{$\langle Z\rangle$}  lies in a
 basis of  \mbox{$\langle\mkern1mu E\mkern2mu\vert\mkern12mu\rangle$},
and, in the affirmative case,
finds such an element.
\end{history}

\begin{content}  What we do in this article is formalize Whitehead's cut-vertex algorithm,
  give a   Bass\d1Serre\d1the\-o\-retic proof of Stong's   cut-ver\-tex lemma,
and  give an algorithm  that yields
  a  basis \mbox{$E''$} of  \mbox{$\langle\mkern1mu E\mkern2mu\vert\mkern12mu\rangle$}
that maximizes~\mbox{$\vert   E'' \cap \langle Z \rangle \,\vert$}.

  In Section\,\ref{sec:autos}, for completeness and to develop the notation and basic results
that will be used, we
 formalize part of Whitehead's discussion of
cut-vertices and  free-group automorphisms, including  his true-word cut-vertex algorithm.

In Section\,\ref{sec:White},    Stong's   beautiful  true-word cut-vertex  lemma   is proved
using edge-cuts of a Cayley tree induced by edge-cuts of a Bass-Serre tree.  At this stage,
 we will have  given a detailed proof
for the polynomial\d1time algorithm for computing a basis of $\Cl(Z)$
that is  more algebraic than  Stong's proof.

In Section\,\ref{sec:CG}, we  restructure the Clifford\d1Goldstein argument using Whitehead's cut-vertex results
in place of the topology,
 and obtain a  slightly faster, more powerful   algorithm that yields
  a  basis \mbox{$E''$} of  \mbox{$\langle\mkern1mu E\mkern2mu\vert\mkern12mu\rangle$}
 which maxi\-mizes \mbox{$\vert   E'' \cap \langle Z \rangle \,\vert$}. In particular,
 \mbox{$  E'' \cap \langle Z \rangle  \ne \emptyset$} if and only if  some element
of  \mbox{$\langle Z\rangle$}  lies in a
 basis of  \mbox{$\langle\mkern1mu E\mkern2mu\vert\mkern12mu\rangle$}.
\end{content}

\section{A formalized cut-vertex algorithm}\label{sec:autos}

This technical section  gives elementary   definitions and arguments
that formalize  part  of
Whitehead's discussion \cite[pp.50--52]{W} of  cut-vertices  and
free\d1group automorphisms.

By a   \textit{graph},
we mean  a set given as the disjoint union of two sets,
called the \textit{vertex-set} and the \textit{edge-set}, together with an
\textit{initial-vertex map} and a \textit{terminal-vertex map}, each of which maps the edge-set to the vertex-set.
 For any set $S$, we write  \mbox{$\mathbb{K}(S)$} to denote the   graph
 which has vertex-set~\hskip-.2pt$S$  and  edge-set \mbox{$S^{\times 2} \coloneq S\,{\times }S$},
where an   edge   \mbox{$(x,y)$} has initial vertex $x$ and terminal vertex~$y$.

\begin{notation}  \label{not:Ds} Recall   Hypotheses~{\normalfont\ref{hyps:1}}.

  $\bullet$  For \mbox{$e \in E$}, we write   \mbox{$\overline e \coloneq e^{-1}$} and
 \mbox{$e^{\pm 1} \coloneq \{e, \overline e\}$}. We write
\mbox{$E^{-1} \coloneq \{ \overline e \mid e \in E\}$}  and
\mbox{$E^{\pm 1} \coloneq E \cup E^{-1}$}. We shall be interested in the   graph 
\mbox{$\mathbb{K}(E^{\pm 1} \cup \{1\})$}, which has basepoint~\mbox{$1$} and 
an inversion map on the vertices.

\medskip

  $\bullet$   Consider any \mbox{$z \in \langle\mkern1mu E\mkern2mu\vert\mkern12mu\rangle$}, and let
   \mbox{$e_1 e_2 \cdots e_n$} represent   the reduced
\mbox{$E^{\pm 1}$}-ex\-pres\-sion for~\mbox{$z$}.  \vspace{-1.5mm}

  $\scriptstyle^{\bullet}$    
 \mbox{$\supp(\{z\} \rel E) = \bigcupp\limits_{i=1}^n  (E \cap e_i^{\pm 1})$} and
  \mbox{$\supp(Z \rel E) = \bigcupp\limits_{z\in Z} \supp(\{z\} \rel E)$}.

  $\scriptstyle^{\bullet}$  We set
 \mbox{$\vert\vert  z \vert \vert_{E} \coloneq n$} and
\mbox{$\vert\vert  Z\vert \vert_{E} \coloneq  \sum\limits_{z\in Z} \vert\vert  z \vert \vert_{E}\mkern2mu$}.

 $\scriptstyle^{\bullet}$  We say that a product \mbox{$xy$} has  \textit{no  \mbox{$E^{\pm 1}$}-cancellation}
 if \mbox{$\vert\vert xy\vert\vert_E =
\vert\vert x\vert\vert_E \hskip2.4pt {+} \hskip2.4pt \vert\vert y\vert\vert_E$}, and then
sometimes write \mbox{$xy$} as \mbox{$x{\cdot}y$} for emphasis.

  $\scriptstyle^{\bullet}$  Suppose that \mbox{$z \ne 1$}.  We set   \mbox{$e_{0} \coloneq  e_{n+1} \coloneq 1$}.  
For  \mbox{$i  \in \{ 0,1,\ldots, n\}$},  we   say that \textit{\mbox{$(e_i, e_{i+1})$} occurs in the
 reduced \mbox{$(E^{\pm 1} \cup \{1\})$}-ex\-pression for $z$}, and note that
  there exist  \mbox{$g', g'' \in   \langle\mkern1mu E\mkern2mu\vert\mkern12mu\rangle$}
such that
\mbox{$ z  =  g\,' {\cdot}  e_i  {\cdot} e_{i+1}  {\cdot} g''$}
with no \mbox{$E^{\pm 1}$}-cancellation,  \mbox{$g'=1$} if \mbox{$e_i = 1$},
and   \mbox{$g''= 1$} if \mbox{$e_{i+1}=1$}.  We set

\smallskip

\centerline{
 \mbox{$\Wh(\{z\}\rel E) \coloneq
E^{\pm1}   \cup   \{1\}    \cup   \{\, (\overline e_i, e_{i+1}) \,\}_{i=0}^n \,\,\subseteq \,\,
\mathbb{K}(E^{\pm 1} \cup \{1\})$}}

\smallskip

\noindent  and \mbox{$\Wh(\{1\}\rel E) \coloneq   E^{\pm1}   \cup   \{1\}$}.
 For example, for each \mbox{$e \in E^{\pm 1}$}, we have
\mbox{$\Wh(\{e\}  \rel E) = E^{\pm1}\cup \{1\}  \cup
  \{   (1, e),  (\overline e, 1) \}$}.  We  also have the pentagonal example
 \mbox{$\Wh(\{x^2y^2\} \rel   \{x,y\})  =  \{ x, y, \overline x,
 \overline y, 1,
(1, x),  (\overline x, x),
 (\overline x,y),   ( \overline y, y),
 (\overline y,1) \}.$}

 \medskip

 $\bullet$  
Let $S$ be a subset of  \mbox{$\langle\mkern1mu E\mkern2mu\vert\mkern12mu\rangle$}.
 If \mbox{$S\ne \emptyset$}, we set 

\centerline{
\mbox{$\Wh(S  \rel E) \coloneq \bigcupp\limits_{z\in S} \Wh(\{z\}\text{ rel }E) \,\,\subseteq \,\,
\mathbb{K}(E^{\pm 1} \cup \{1\})$},}

\noindent and we set
 \mbox{$\Wh( \emptyset \rel E) \coloneq  E^{\pm1}   \cup   \{1\}$}.
 In \mbox{$\Wh(S\text{ rel }E)$}, a vertex \mbox{$e_\star$}  
is said to be a \textit{Whitehead cut-vertex}  if  removing \mbox{$e_\star$} and  all  the edges 
incident to~\mbox{$e_\star$}
leaves a basepointed  graph that is not connected; this entails  \mbox{$e_\star \ne 1$}.
If \mbox{$\Wh(S\text{ rel }E)$} is not connected, then each element of \mbox{$E^{\pm 1}$} is a
Whitehead cut-vertex, since the set of valence-zero vertices is closed under inversion.

 \medskip

 $\bullet$  We let  \mbox{$\cuts(E)$} denote the
set of those ordered triples  \mbox{$(\,\null_0D,\null_1 D, e_\star\,)$} such that
 \mbox{$\null_0D  \cup \null_1D  = E^{\pm 1}$}, \, \mbox{$\null_0D  \cap \null_1D   = \{e_\star\}$},
and \mbox{$\null_1D  \ne \{e_\star \}$}.  Clearly, 
\mbox{$\null_0 D  \subseteq E^{\pm 1}$}, \mbox{$\null_1 D  \subseteq E^{\pm 1}$}, and
  \mbox{$e_\star\, \in E^{\pm 1}$}. Suppose that \mbox{$\mathbf{C} = (\,\null_0D,\null_1D, e_\star\,)  \in \cuts(E)$}.

  $\scriptstyle^{\bullet}$  For each \mbox{$(\alpha,\beta) \in \{0,1\}^{\times 2}$},
we   set
\mbox{$\null_\alpha D_\beta \coloneq \null_\alpha D  \cap \null_\beta D ^{-1}$} and
 \mbox{$\null_\alpha E_\beta \coloneq  E \cap\, \null_\alpha D_\beta$}.

 $\scriptstyle^{\bullet}$   
 Let  \mbox{$\chi \colon E^{\pm 1} \to \{0,1\}$},   \mbox{$e \mapsto   \chi(e) \coloneq \vert \{e\} \cap \null_1D\vert$},
be the characteristic map of~\mbox{$\null_1D$}.
We set \mbox{$\eta_{\mathbf{C}} \coloneq \chi(\overline e_\star\,)\in \{0,1\}$}
 and  \mbox{$d_\star\, \coloneq e_\star^{2\eta_{\mathbf{C}}{-}1}  \in\, e_\star^{\pm 1}$}, that is,
 \mbox{$d_\star = \overline e_\star$} if \mbox{$\overline e_\star \in \null_0D$},
while \mbox{$d_\star = e_\star$} if \mbox{$\overline e_\star \in \null_1D$}. 
We define \mbox{$\phi_{\mathbf{C}}$} to be the automorphism   of  
\mbox{$\langle\mkern1mu E\mkern2mu\vert\mkern12mu\rangle$}
that fixes \mbox{$d_\star$}  and  maps $e$ to
\mbox{$d_\star^{\mkern3mu\chi(e)} \,e \,  \overline d_\star \null^{\mkern2mu\chi(\overline e)} $}
for each \mbox{$e \in E {-}d_\star^{\pm 1}$}.

   $\scriptstyle^{\bullet}$    We define   three subgraphs of
 \mbox{$\mathbb{K}(E^{\pm 1} \cup \{1\})$}:
\mbox{$\null_0\mkern-5mu\Who(\mathbf{C}) \coloneq   \mathbb{K}(\null_0D \cup \{1\})$};
\mbox{$\null_1\mkern-5mu\Who(\mathbf{C}) \coloneq   \mathbb{K}(\null_1D)$}; and,
\mbox{$\Wh(\mathbf{C}) \coloneq  \mathbb{K}(\null_0D \cup \{1\}) \cup  \mathbb{K}(\null_1D)$}.
We say that  \mbox{$\mathbf{C}$} \textit{cuts}  each subgraph  of \mbox{$\Wh(\,\mathbf{C})$} 
 with the full   vertex-set, \mbox{$E^{\pm1}   \cup   \{1\}$}. 

If  \mbox{$\mathbf{C}$} cuts   \mbox{$\Wh(Z\text{ rel }E)$},
then     \mbox{$e_\star$} is  a Whitehead cut-vertex of   \mbox{$\Wh(Z\text{ rel }E)$},
since   \mbox{$ \null_0D \cup \{1\}$} and
 \mbox{$\null_1D$} have union
 \mbox{$E^{\pm1}   \cup   \{1\}$} and intersection
 \mbox{$\{e_\star \}$},  while
 \mbox{$\null_0D \cup \{1\}  \ne \  \{e_\star \}  \ne  \null_1D$}.
\end{notation}

\begin{lemma}\label{lem:W}  With Hypotheses~{\normalfont\ref{hyps:1}}, fix
\mbox{$\mathbf{C} = (\null_0D, \null_1D, e_\star\,) \in \cuts(E)$}, and let
\mbox{$z \in \langle\mkern1mu E\mkern2mu\vert\mkern12mu\rangle$}.
Then the following hold.  \vspace{-2mm}
 \begin{enumerate}[{\rm (i)}]
\setlength\itemsep{-2pt}
\item  \mbox{$\tilde{E} \coloneq \hskip -15pt \bigcupp\limits_{(\alpha,\beta)\in\{0,1\}^{\times 2}}
\hskip -10pt ( a^{\,\alpha}(\,\null_\alpha E_\beta)\, \overline a^\beta)$} is a basis of
  \mbox{$\langle\mkern1mu E\cup\{a\}\mkern2mu\vert\mkern12mu\rangle$}. \vspace{-.5mm}
\item   \mbox{$\vert \vert z \vert \vert_{\widetilde{E}}  =
\vert \vert z \vert \vert_{ E\phantom{\widetilde{E}}}\mkern-8mu$}  if and only if \mbox{$\mathbf{C}$}
cuts \mbox{$\Wh(\{z\} \rel E)$}.
\item If  \mbox{$\mathbf{C}$}
cuts \mbox{$\Wh(\{z\} \rel E)$}, then
\mbox{$\vert \vert z^{\mkern2mu\overline \phi_{\mathbf{C}}} \vert \vert_{E} \le   \vert \vert z \vert \vert_{E}$},
\item If  \mbox{$\mathbf{C}$}
cuts \mbox{$\Wh(\{z\} \rel E)$} and \mbox{$e_\star$} has positive valence in the subgraph
\mbox{$\Wh(\{z\} \rel E) \cap \null_{1-\eta_{\mathbf{C}}}\mkern-7mu\Who(\mathbf{C})$}, then
\mbox{$\vert \vert z^{\mkern2mu\overline \phi_{\mathbf{C}}} \vert \vert_{E} <  \vert \vert z \vert \vert_{E}$}.
\end{enumerate}
\end{lemma}

\begin{proof} Set \mbox{$F \coloneq \langle\mkern1mu E\mkern2mu\vert\mkern12mu\rangle$},
 \mbox{$\tilde F \coloneq \langle\mkern1mu E\cup\{a\}\mkern2mu\vert\mkern12mu\rangle$},
\mbox{$\eta \coloneq \eta_{\mathbf{C}}$},  and \mbox{$\phi \coloneq \phi_{\mathbf{C}}$}.

 \hskip10pt(i).  Recall that  \mbox{$e_\star  \in E^{\pm 1}$}.

  If \mbox{$e_\star\in E$}, then  \mbox{$\null_0 E_\eta \cap \null_1E_\eta  = \{e_\star \}$} \vspace{.5mm} and there
are no other overlaps among the~\mbox{$\null_\alpha E_\beta$}.
Since
\mbox{$\{ a^0 e_\star \overline a^{\mkern4mu\eta}, a^1 e_\star \overline a^{\mkern4mu\eta}\}\subseteq \tilde E$} and
\mbox{$(a^1 e_\star \overline a^{\mkern4mu\eta})(a^0 e_\star \overline a^{\mkern4mu\eta})^{-1} = a$},
we see easily that \mbox{$\tilde E$}  is a basis of~\mbox{$\tilde F$}.

Similarly, if \mbox{$e_\star  \in E^{-1}$}, then
\mbox{$\null_\eta E_0 \cap \null_\eta E_1 = \{\overline e_\star \}$}\vspace{1mm} and there
are no other overlaps among the~\mbox{$\null_\alpha E_\beta$}.  Again,
 \mbox{$\tilde E$}  is a basis of~\mbox{$\tilde F$}.

(ii).  Let  \mbox{$e_1 e_2 \cdots e_n$} represent the reduced \mbox{$E^{\pm 1}$}-expression for  $z$.
For any   map  \mbox{$  \{0,\ldots, n\} \to \{0,1\}$},  \mbox{$i \mapsto \chi_i$},
the following three conditions are easily seen to be equivalent.

\hskip 10pt $\bullet$   the reduced
\mbox{$\tilde E^{\pm 1}$}-expression for  $z$  is
  \mbox{$(a^{\mkern1mu\chi_0} e_1 \overline  a^{\mkern3mu\chi_1}) ( a^{\mkern1mu\chi_1} e_2
\overline  a^{\mkern3mu\chi_2})
\cdots  (a^{ \chi_{n-1}} e_n \overline a^{\mkern3mu\chi_n})$}.

\hskip 10pt $\bullet$   
\mbox{$e_i \in \null_{\chi_{i-1}} D_{\chi_i}$},  \mbox{$i = 1,2,\ldots, n$}, and  \mbox{$\chi_0 = \chi_{n} = 0$}.

\hskip 10pt $\bullet$  
 \mbox{$(\overline e_i,  e_{i+1}) \in  \null_{\chi_i}D^{\times 2}$},    \mbox{$i = 1,2,\ldots, n{-}1$}, 
 \mbox{$(\overline e_n,  e_{1}) \in  \null_{0}D^{\times 2}$},  and \mbox{$\chi_0 = \chi_{n} = 0$}.

\noindent Now (ii) follows.   \vspace{.5mm}

(iii). Let  \mbox{$\tilde \phi \colon \tilde F \to F$} denote the retraction that
carries  $a$~to \mbox{$d_\star\, \coloneq \,\,e_\star^{2\eta-1}$}. We apply  \mbox{$\tilde \phi$}
to \mbox{$\{a^\eta e_\star\, \overline  a^{\mkern3mu \eta}, a^{1-\eta} e_\star\, \overline  a^{\mkern3mu\eta}\}
= \{a^0 e_\star\, \overline a^{\mkern3mu \eta} , a^1 e_\star\, \overline  a^{\mkern3mu \eta}\} \subseteq \tilde E^{\pm 1}.$}
Here, we have
\mbox{$(a^\eta e_\star\, \overline  a^{\mkern3mu \eta})^{\tilde \phi} =
e_\star\, d_\star^{\mkern3mu \eta-\eta} = e_\star\,$} and
\mbox{$(a^{1-\eta} e_\star\, \overline  a^{\mkern3mu\eta})^{\tilde \phi}  =
e_\star\, d_\star^{1-\eta  -\eta} =  1$}.
It follows that \mbox{$\tilde \phi$} carries \mbox{$\tilde E$} to \mbox{$E^{\mkern1mu\phi} \cup \{1\}$}.
Since \mbox{$z^{\tilde \phi} = z$}, we see that
\mbox{$\vert\vert z \vert\vert_{E^{\phi}\phantom{\widetilde{E}\mkern-12mu}}
\le \vert\vert z \vert\vert_{\widetilde{E}}$}.  Now
\mbox{$\vert \vert z^{\mkern1mu\overline \phi} \vert \vert_{E\phantom{\widetilde{E}\mkern-12mu}}
= \vert \vert z \vert \vert_{E^{\,\phi}\phantom{\widetilde{E}\mkern-12mu}} \le \vert \vert z \vert \vert_{\widetilde{E}}
= \vert \vert z \vert \vert_{E\phantom{\widetilde{E}\mkern-12mu}}$}, by (ii). \vspace{.5mm}

 (iv).  There exists some vertex \mbox{$e$} of \mbox{$\null_{1-\eta}\mkern-5mu\Who(\mathbf{C})$} such that
  \mbox{$(\overline e\mkern2mu, e_\star)$}
 occurs in the reduced
\mbox{$(E^{\pm 1} \cup \{1\})$}-ex\-pres\-sion for \mbox{$ z $} or \mbox{$ \overline z $}.
Necessarily,  \mbox{$e \ne e_\star\,$}.  Hence,  \mbox{$e \not \in \null_\eta \mkern-7mu\Who(\mathbf{C})$}.
As in (ii), the element \mbox{$( a^{1-\eta}\,   e_\star\, \overline a^{\mkern3mu \eta})  \in \tilde E^{\pm 1}$} 
occurs in the reduced
\mbox{$\tilde E^{\pm 1}$}-ex\-pression for  $z$  or~\mbox{$\overline z$}.
Hence, \mbox{$( a^{1-\eta}   e_\star\, \overline a^{\mkern3mu \eta})$} or
\mbox{$( a^{\eta} \overline e_\star\, \overline a^{\mkern3mu 1-\eta}\, )$}  occurs in the reduced
\mbox{$\tilde E^{\pm 1}$}-expression for  $z$.
As in (iii), each such term is mapped to $1$ by  \mbox{$\tilde \phi$}.
Thus, \mbox{$\vert \vert z^{\mkern1mu\overline \phi} \vert \vert_{E\phantom{\tilde E}}\mkern-8mu
= \vert \vert z \vert \vert_{E^{\,\phi}\phantom{\tilde E}}\mkern-12mu < \vert \vert z \vert \vert_{\widetilde{E}}
 = \vert \vert z \vert \vert_{E\phantom{\tilde E}}\mkern-8mu$}.
\end{proof}

\begin{algorithm}\label{alg:preW}   Recall   Hypotheses~{\normalfont\ref{hyps:1}}. Whitehead's
\textit{cut-vertex subroutine} \mbox{\cite[p.\,51]{W}} has the following structure.

\noindent \textsc{Input:}   A Whitehead cut-vertex \mbox{$e_\star$} of  \mbox{$\Wh(Z \rel E_Z)$}.

 \noindent \textsc{Output:}    A \mbox{$\mathbf{C}   \in \cuts(E)$}
with   \mbox{$\Wh(Z \hskip-.8pt \rel \hskip-.8pt  E) \subseteq \Wh( \mathbf{C})$}
and   \mbox{$\vert\vert Z^{\mkern2mu\overline \phi_{\mathbf{C}}}\vert\vert_{E} <
  \vert\vert Z \vert\vert_{E}$}.

\noindent \textsc{Procedure.}  We consider two cases.

\smallskip

\noindent \textsc{Case 1:} \mbox{$\Wh(Z \rel E_Z)$} is connected.

Deleting  \mbox{$e_\star$} and its incident edges from \mbox{$\Wh(Z \rel E_Z)$}
leaves a subgraph  that    has a unique expression  as the disjoint union of
two nonempty subgraphs \mbox{$X_0$} and \mbox{$X_1$} such that  \mbox{$X_0$} is connected
and contains~\mbox{$\{1\}$}.

Set
 \mbox{$\null_0 D \coloneq (X_0 \cap E_Z^{\pm 1}) \cup \{e_\star\} \cup (E^{\pm 1}{-}E_Z^{\pm 1})$},
\hskip 4pt  \mbox{$\null_1 D \coloneq (X_1 \cap E_Z^{\pm 1}) \cup \{e_\star\}$}, \hskip 4pt and
\mbox{$\mathbf{C}  \coloneq (\,\null_0D,\null_1D, e_\star)   \in \cuts(E)$}.
Then   \mbox{$\Wh(Z \rel E) \subseteq \Wh(\,\mathbf{C})$}, and
 \mbox{$e_\star$} has positive valence in both
 \mbox{$\Wh(Z \rel E) \cap  \null_{0}\mkern-5mu\Who(\mathbf{C})$}
and \mbox{$\Wh(Z \rel E) \cap  \null_{1}\mkern-5mu\Who(\mathbf{C})$}.  Thus, \mbox{$e_\star$} has
positive valence in \mbox{$\Wh(Z \rel E) \cap  \null_{1-\eta_{\mathbf{C}}}\mkern-7mu\Who(\mathbf{C})$}.
It follows from Lemma~\ref{lem:W}(iii),(iv) that
 \mbox{$\vert\vert Z^{\mkern2mu\overline \phi_{\mathbf{C}}}\vert\vert_{E} <
  \vert\vert Z \vert\vert_{E}$}.
We return \mbox{$\mathbf{C}$}
and   terminate the procedure.

\smallskip

 \noindent \textsc{Case 2:} \mbox{$\Wh(Z \rel E_Z)$} is not connected.

Let $X$ denote the component of \mbox{$\Wh(Z \rel E_Z)$} containing~$\{1\}$,
and let \mbox{$D \coloneq X \cap E_Z^{\pm 1}$}.   If it were the case that
\mbox{$D^{-1} =  D$}, then it is not difficult to see that we would have
\mbox{$Z \subseteq \langle   D \rangle$},   \mbox{$E_Z^{\pm 1} =   D$}, and
  \mbox{$X = \Wh(Z\text{ rel }E_Z)$}, which would contradict
the assumption that \mbox{$\Wh(Z\text{ rel }E_Z)$} is not connected.
Thus, \mbox{$D^{-1} \ne  D$},  \mbox{$D\not \subseteq  D^{-1}$}, and \mbox{$D {-}  D^{-1}  \ne \emptyset$}.

Choose     \mbox{$e_\star'\hskip-2.2pt \in \hskip-2.2pt D{-}D^{-1}$}, and set 
  \mbox{$\null_0D \coloneq D \cup (E^{\pm 1}{-}E_Z^{\pm 1})$},
\mbox{$\null_1D \coloneq (E_Z^{\pm 1}{-}D)  \cup \{e_\star'\}$},  
and \mbox{$\mathbf{C}  \coloneq (\,\null_0D,\null_1D, e_\star')   \in \cuts(E)$}. 
It is clear that
\mbox{$\Wh(Z \rel E) \subseteq \Wh(\,\mathbf{C})$}.
Here,  \mbox{$\eta_{\mathbf{C}} = \vert \{ (e_\star')^{-1} \,\} \cap \null_1D  \vert = 1$}.
Also, \mbox{$\Wh(Z \rel E) \cap \,\null_0\mkern-5mu \Who(\mathbf{C}) \supseteq  X$,}  the
  component  of  \mbox{$\Wh(Z \rel E_Z)$} that contains
 \mbox{$\{e_\star', 1\}$}.  Since \mbox{$e_\star'$} has positive
valence  in~\mbox{$X$}, it follows from Lemma~\ref{lem:W}(iii),(iv)   that
 \mbox{$\vert \vert Z^{\mkern2mu\overline \phi_{\mathbf{C}}} \vert \vert_E
 < \vert \vert Z  \vert \vert_E$}.
We return \mbox{$\mathbf{C}$} and terminate the procedure.  \qed
\end{algorithm}

\begin{algorithm}\label{alg:cut}  Recall   Hypotheses~{\normalfont\ref{hyps:1}}.
Via the  mock flow chart  \vspace{3mm}

\noindent\null\hskip16pt\setlength\fboxrule{1pt}\fbox{Set
\mbox{$\Phi \coloneq 1 \in \Aut\langle\mkern1mu E\mkern2mu\vert\mkern12mu\rangle$}
and \mbox{$Z\mkern1mu' \coloneq Z$}.} \vspace{1mm}

\noindent\null\hskip40pt $\downarrow$ \vspace{1mm}

\noindent$\to$\hskip4pt\fbox{Find \mbox{$E_{Z'}\coloneq \supp(Z\mkern1mu'  \rel E)$}
and construct \mbox{$\Wh(Z\mkern1mu' \rel E_{Z'})$}.} \vspace{1mm}

\noindent\null\hskip40pt $\downarrow$  \vspace{1mm}

\noindent\null\hskip16pt\fbox{Search for a  Whitehead cut-vertex $e_\star$ of
\mbox{$\mbox{$\Wh(Z\mkern1mu'\rel E_{Z'})$}$}.} \vspace{1.3mm}

\noindent\null\hskip40pt $\downarrow$  \vspace{.7mm}

\noindent\null\hskip16pt \setlength\fboxrule{.5pt}\mbox{\fbox{Does such an \mbox{$e_\star$}  exist?} \setlength\fboxrule{1pt}
 $\xrightarrow{\textsc{No}}$    
\fbox{Return  \mbox{$(\Phi,    Z\mkern1mu')$}.} $\to$ \vspace{1mm} \fbox{Stop.}} 

\noindent\null\hskip40pt $\null_{\textstyle \downarrow\mkern-5mu\null^{\null_{\scriptstyle\textsc{Yes}}}}$ \vspace{1mm}

\noindent\null\hskip16pt\setlength\fboxrule{1pt}\fbox{Algorithm\,\ref{alg:preW}  yields a
 \mbox{$\phi  \in   \Aut \langle\mkern1mu E\mkern2mu\vert\mkern12mu\rangle$} such that
 \mbox{$\vert \vert {Z\mkern1mu'}^{\mkern2mu\overline \phi} \vert \vert_{E}   <
\vert \vert Z\mkern1mu' \vert \vert_{E}$}.}   \vspace{1mm}

\noindent\null\hskip40pt $\downarrow$ \vspace{1mm}

\noindent\hskip-1pt$\leftarrow$\hskip5pt\fbox{Reset  \mbox{$\Phi  \coloneq  \phi{\cdot}\Phi$} and
 \mbox{$Z\mkern1mu'  \coloneq  Z\mkern1mu'^{\mkern2mu\overline \phi}$}, thereby decreasing
 \mbox{$\vert \vert Z\mkern1mu' \vert \vert_{E}$}.}

\vspace{-58.5mm}

\noindent \rule[0pt]{.6pt}{5.45cm} \vspace{7.5mm}

\noindent  White\-head's \textit{cut-vertex algorithm} \mbox{\cite[p.\,51]{W}}  returns
a pair \mbox{$(\Phi,  Z\mkern1mu')$} such that
 \mbox{$\Phi   \in   \Aut \langle\mkern1mu E\mkern2mu\vert\mkern12mu\rangle$},
   \mbox{$Z\mkern1mu'  =  Z^{\,\overline \Phi}$}, and the isomorphic graphs
  \mbox{$\Wh(Z\mkern1mu' \rel \supp(Z\mkern1mu' \hskip-1.3pt  \rel \hskip-1.3pt  E))$}
and \mbox{$\Wh(Z  \rel \supp(Z  \rel E^{\,\Phi}))$}
have no Whitehead cut-vertices. It is then not difficult to find \mbox{$\supp(Z\mkern1mu'   \rel   E)$},
 \mbox{$E^{\,\Phi}$},  and, hence,  \mbox{$\supp(Z  \rel E^{\,\Phi})$}.

 Information about  these will be
given in Lemmas~\ref{lem:cvl} and~\ref{lem:sto}.  For example,
 \mbox{$\vert \supp(Z  \rel E^{\,\Phi}) \vert$} is smallest-possible over all bases of
\mbox{$\langle\mkern1mu E\mkern2mu\vert\mkern12mu\rangle$}, that is,
 \mbox{$\supp(Z  \rel E^{\,\Phi})$} is a basis of \mbox{$\Cl(Z)$}.
Also,  \mbox{$Z$}~is a sub-basis of~\mbox{$\langle\mkern1mu E\mkern2mu\vert\mkern12mu\rangle$}
if and only if \mbox{$Z {\,\cap\,} Z^{-1} \hskip-.5pt = \hskip-.5pt \emptyset$}
and \mbox{$Z\mkern1mu' \subseteq E^{\pm 1}$};
in this event,
 \mbox{$Z \cup (E^{\,\Phi}{-} Z^{\pm1})$} is a
basis of~\vspace{2mm}\mbox{$\langle\mkern1mu E\mkern2mu\vert\mkern12mu\rangle$}. \qed
\end{algorithm}

\begin{notes}\label{notes:PT}
Although Whitehead did not mention it, it is possible to implement Algorithm\,\ref{alg:cut}
in such a way that it terminates in time that is
polynomial (linear?) in  \mbox{$\vert E \vert + \vert\vert Z \vert\vert_E$}.
Depth-first searches may be used to find the  component $X$
of \mbox{$\Wh(Z\mkern1mu' \rel E_{Z'})$} that contains~\mbox{$\{1\}$},
and to search for an element \mbox{$e_\star \in X \cap E_{Z\mkern1mu'}^{\pm 1}$}
such that either  \mbox{$\overline e_\star  \not \in X$} or
 removing  \mbox{$e_\star$} and its incident edges from  \mbox{$X$} leaves a graph that is
not connected.  If no such  \mbox{$e_\star$} exists then
\mbox{$\Wh(Z\mkern1mu' \rel E_{Z'})$} has no Whitehead cut-vertices, as was seen
in Algorithm\,\ref{alg:preW}.
If such an \mbox{$e_\star$} exists, then
it may be used to construct a~$\phi$
such that \mbox{$\vert \vert Z\mkern1mu'^{\mkern2mu\overline \phi} \vert \vert_{E}  <
\vert \vert Z\mkern1mu' \vert \vert_{E}$},   as was also seen in Algorithm\,\ref{alg:preW}.
\end{notes}

 \section{Bass-Serre proofs of cut-vertex lemmas}\label{sec:White}

\begin{review}\label{rev:trees} Let $F$ be a group.

$\bullet$ Let $S$ be a subset of $F$.  We let \mbox{$\Cayley(F,S)$} denote  the graph with
 vertex-set~$F$ and edge-set \mbox{$F{\times}S$},
where each edge  \mbox{$(g,s) \in F{\times}S$} has initial vertex~\mbox{$g$} and terminal  vertex~\mbox{$gs$};
we  shall sometimes write  \mbox{$\operatorname{edge}(g \xrightarrow{\bullet (s)} gs)$} to denote the pair
 \mbox{$(g,s)$} viewed as an edge.
Then \mbox{$\Cayley(F,S)$} is an   $F$-graph.
It is a tree when \mbox{$S$} is a basis of $F$.  See, for example,  \cite[Theorem~I.7.6]{DD}.

$\bullet$ Let $I$ be a set  and
   \mbox{$(H_i)_{i\in I}$}  be a family of subgroups of $F$.
We let \mbox{$\BST(F,(H_i)_{i\in I})$} denote the  graph whose vertex-set is the disjoint union of
the set~$F$ together with the sets \mbox{$F/H_i$}, \mbox{$i \in I$}, and whose edge-set is $F{\times}\mkern2mu I$,
where each
edge  \mbox{$(g, i) \in F{\times}\mkern2mu I$} has initial vertex
\mbox{$g$} and terminal  vertex~\mbox{$gH_i$}; we  shall sometimes write
\mbox{$\operatorname{edge}(g \xrightarrow{\bullet (H_i)} gH_i)$} to denote the pair
\mbox{$(g,i)$} viewed as an edge.
 Then \mbox{$\BST(F, (H_i)_{i\in I})$} is an $F$-graph.
It is a tree when \mbox{$F = \bigast\limits_{i \in I} H_i$}, by  a
result of H.\,Bass and J.-P.\,Serre.  See, for example, \cite[Theorem~I.7.6]{DD}.
 \end{review}

Notice that if a subset \mbox{$S$}
 of~\mbox{$\langle\mkern1mu E\mkern2mu\vert\mkern12mu\rangle$}
  contains \mbox{$E$}, then    \mbox{$\Wh(S \rel E)$}
contains the basepointed star   \mbox{$\Wh(E \rel E)$}, and therefore has no Whitehead cut-vertices.
The following amazing partial converse can be extracted from
 the (1){$\Rightarrow$}(3) part of \cite[Theorem~10]{Stong}.
The case where each free-product factor is cyclic is essentially Whitehead's cut-vertex lemma \cite[Lemma]{W}.

\begin{sw}\label{thm:sto}
 For each finite set $E$ and free-product factorization \vspace{-1.5mm}
\mbox{$\langle\mkern1mu E\mkern2mu\vert\mkern12mu\rangle = \bigast\limits_{i\in I}  H_i$}
such that  \mbox{$\bigcupp\limits_{i\in I} \hskip-.5pt H_i \not \supseteq E$}, the graph
\mbox{$\Wh(\,(\,\bigcupp\limits_{i\in I}  H_i) \rel E)$} has a Whitehead cut-vertex.
\end{sw}

\begin{proof}   Set
\mbox{$ F \coloneq  \langle\mkern1mu E\mkern2mu\vert\mkern12mu\rangle
=   \bigast\limits_{i\in I}  H_i$}. Recall Review\,\ref{rev:trees},
and set \mbox{$T \coloneq \Cayley( F,  E)$} 
and  \mbox{$T^\ast \coloneq \BST( F, (H_i)_{i\in I})$}.
Thus,  \mbox{$T$} and \mbox{$T^\ast$} are
  \mbox{$F$}-trees whose vertex-sets  contain  \mbox{$ F$}.

We work first with \mbox{$T^\ast$}.  
We let \mbox{$\operatorname{link\,}_{T^\ast}(1)$} denote the set of 
\mbox{$T^\ast$}-edges incident to the \mbox{$T^\ast$}-vertex $1$,
and   \mbox{$\operatorname{star\,}_{T^\ast}(1)$} denote the set of components of the forest
\mbox{$T^\ast{-}\operatorname{link\,}_{T^\ast}(1)$}.  
For each \mbox{$T^\ast$}-vertex $v$, there exists a unique component 
 \mbox{$\chi(v) \in \operatorname{star\,}_{T^\ast}(1)$} such that
\mbox{$v \in \chi(v)$}.   For any \mbox{$T^\ast$}-vertices  $v$~and $w$, we let
 \mbox{$T^\ast[v,w]$} denote  the    $\subseteq$-smallest subtree
of \mbox{$T^\ast$} that contains $\{v,w\}$, and then
 \mbox{$\chi(v) \ne \chi(w)$} if and only if  \mbox{$1 \in  T^\ast[v,w]$} and \mbox{$v \ne w$}.
Also, $\chi$  restricts to a\vspace{-.5mm} map \mbox{$F \to \operatorname{star\,}_{T^\ast}(1)$}.  

In $T$ now,  set \mbox{$\delta \coloneq \{  \operatorname{edge}(g \xrightarrow{\bullet(e)} ge)=(g,e)  \in
 F {\times}   E  \subseteq T  \mid  \chi(g)\hskip-.8pt \ne \hskip-.8pt \chi(ge) \}$}. Clearly,
\mbox{$\chi$} is constant on the vertex-set  of each component  of  \mbox{$T  {-}\delta $}.
An element \mbox{$(g,e) \in F{\times}\mkern1mu E$} lies  in $\delta$
 if and only if
\mbox{$1 \in T^\ast[g,ge]$}, or,  equivalently, \mbox{$\overline g \in T^\ast[1,e]$}.
Since  $E$ is nonempty and  finite,   it is clear that $\delta$ is nonempty and finite.
Hence, there exists  
  \mbox{$(g_\delta,e_\delta) \in F{\times}E^{\pm 1}$} satisfying
\mbox{$\vert \vert g_\delta e_\delta \vert \vert_{E} \hskip-.3pt =\hskip-.3pt
\vert \vert g_\delta \vert \vert_{E}{+}1$}
and  \mbox{$\chi(g_\delta) \ne \chi(g_\delta e_\delta)$} such that
  \mbox{$\vert \vert g_\delta\vert \vert_{{E}}$} has the maximum possible value.

We shall now show
 that  \mbox{$g_\delta  \ne  1$}.
By hypothesis, there exists   \mbox{$e_0 \in E {-} \bigcupp_{I}  H_i$}.
Since \mbox{$e_0 \ne 1$}, there exists some \mbox{$T^\ast[1,e_0]$}-neighbour
 of~$1$, necessarily   \mbox{$1H_{i_0}$}  for some \mbox{$i_0 \in I$}.
Clearly \mbox{$e_0 \ne 1H_{i_0}$}; thus, there exists some \mbox{$T^\ast[1,e_0]$}-neigh\-bour of
\mbox{$1H_{i_0}$} other than $1$, necessarily\vspace{.5mm}  some \mbox{$h_0 \in H_{i_0}{-}\{1\}$}.
Now \mbox{$h_0 \in T^{\ast}[1,e_0]$},     \mbox{$1  \in  T^{\ast}[\overline h_0,\overline h_0e_0]$},
\mbox{$\chi(\overline h_0) \ne \chi(\overline h_0e_0)$},
and \vspace{.5mm}\mbox{$\vert\vert g_\delta \vert \vert_E \ge \min \{\vert\vert \overline h_0 \vert \vert_E,
\vert\vert \overline h_0 e_0\vert \vert_E \}$}. We know that \mbox{$e_0   \in  E{-}  \{h_0\}$}
and \mbox{$h_0 \in  H_{i_0}{-}\{1\}$}.
Hence, \mbox{$1 \not \in \{\overline h_0, \overline h_0 e_0\}$}  and
 \mbox{$g_\delta  \ne  1$}.
There exists a unique    \mbox{$e_\star \in   E^{\pm 1}$} such that
\mbox{$\vert \vert g_\delta  e_\star  \vert \vert_{ {E}} =
\vert \vert g_\delta    \vert \vert_{ {E}}{-}1 $}. Clearly,  \mbox{$e_\star \not \in \{1, e_\delta\}$}.

Let us review the graph of interest.
In $T$, define \mbox{$\operatorname{link\,}_T(1)$} and 
\mbox{$\operatorname{star\,}_T(1)$} as for \mbox{$T^\ast$}.  
 For each \mbox{$e \in E^{\pm 1} \hskip-.3pt\cup \{1\}$},
there exists a unique component \mbox{$[e] \in \operatorname{star\,}_T(1)$} such that \mbox{$e \in [e]$}.
Then the map \mbox{$E^{\pm 1} \cup \{1\} \to \operatorname{star\,}_T(1)$}, \mbox{$e \mapsto [e]$},
is bijective. 
Fix  an edge  \mbox{$(e', e'')$} of
\mbox{$\Wh( \,\bigcupp_{I}  H_i   \rel E)$}.  Here,  there exist
\mbox{$i \in I$},   \mbox{$h \in  H_i {-}\{1\}$}, and    \mbox{$g', g'' \in    F$}
such that \mbox{$h  = \overline g\,' {\cdot} \overline e\,'  {\cdot} e''  {\cdot} g''$}
with no \mbox{$E^{\pm 1}$}-cancellation,
 \mbox{$g'=1$} if \mbox{$e' = 1$}, and   \mbox{$g''= 1$} if \mbox{$e''=1$}.
Thus,  \mbox{$ e' g' h    = e'' g''$},   \mbox{$ e' g'  H_i   =     e'' g''   H_i$},  
\mbox{$e'{\cdot}g' \in [e']$}, and \mbox{$e''{\cdot}g'' \in [e'']$};
it may happen that   \mbox{$e' = 1 = g'$} and \mbox{$[e'] = \{1\}$}.
 
We now return to  \mbox{$g_\delta$} and  \mbox{$e_\star$}. We see that \mbox{$1 \in g_\delta[e_\star]$},
 \mbox{$\delta \subseteq g_\delta \operatorname{link\,}_T(1) \cup g_\delta [e_\star]$},
and   $\chi$ is constant on the vertex-set of each component of 
\mbox{$T{-}(g_\delta \operatorname{link\hskip.8pt}_T(1) \cup g_\delta[e_\star])$}.
We shall   show that if \mbox{$e_\star  \not  \in \{e', e''\}$},
 then \mbox{$ \chi(g_\delta   e' ) =    \chi(g_\delta   e'' )$}.
As \mbox{$e' \ne e_\star$}, we see   \mbox{$1 \not \in g_\delta[e']$} 
and  $\chi$ maps the vertex-set of~\mbox{$ g_\delta[e']$} 
to \mbox{$\{\chi(g_\delta e')\}$}.
As \mbox{$ g_\delta     e' g' \in g_\delta[e']$}, we see
 \mbox{$\operatorname{edge}(g_\delta   e'g'\xrightarrow{\bullet(H_i)} g_\delta   e'g'H_i)
 \not \in  \operatorname{link\,}_{T^\ast}(1)$} and    \mbox{$ \chi(g_\delta   e'g') = \chi(g_\delta   e')$}. 
It follows  that  
 \mbox{$\chi(g_\delta   e')  =  \chi(g_\delta   e'g')
= \chi(g_\delta   e'g' H_i )= \chi(g_\delta   e''g'' H_i )  
= \chi(g_\delta   e''g'')  = \chi(g_\delta e'')$}. 

Let $W$ denote the graph that is obtained  from 
 \mbox{$\Wh(  \,\bigcupp_{I}  H_i  \rel E)$}  by removing~\mbox{$e_\star$}
and its incident edges.
We have now proved that
\mbox{$\chi(g_\delta   -)$} is constant on the vertex-sets  of the components of~$W$.
Since \mbox{$1$}~and
 \mbox{$e_\delta$} are vertices of
 \mbox{$W$}
such that   \mbox{$\chi(g_\delta 1)  \ne  \chi(g_\delta  e_\delta)$},
 we see that  $W$~is not connected, and, hence,  \mbox{$e_\star$}
is a Whitehead cut-vertex   of~\mbox{$\Wh(  \,\bigcupp_{I}  H_i  \rel E)$}.
  \end{proof}

\begin{corollary}\label{cor:sto}
With  Hypotheses~{\normalfont\ref{hyps:1}}, suppose that
\mbox{$\Wh(Z \rel E)$} has no Whitehead cut-ver\-tices.  For each 
 \vspace{-1.5mm} free-product factorization
\mbox{$\langle\mkern1mu E\mkern2mu\vert\mkern12mu\rangle = \bigast\limits_{i\in I}  H_i$}
 such that
\mbox{$Z \subseteq  \bigcupp\limits_{i\in I}  H_i$}, the set 
 $E$ contains a  basis of \vspace{-1.5mm} each~\mbox{$H_i$}.
\end{corollary}

\begin{proof}
Let $i$ range over $I$.   Set \mbox{$E_i \coloneq E \cap H_i$}.
Then the \mbox{$E_i$}   are pairwise disjoint.
As it contains  \mbox{$\Wh(Z \rel E)$},
 \mbox{$\Wh( \,\bigcupp_I H_i   \rel E)$}   has no Whitehead cut-ver\-tices.
By the contrapositive of Theorem~\ref{thm:sto}, \mbox{$E \subseteq \bigcupp_{I} H_i$}.
Thus,  \mbox{$E  =  \bigcupp_{I} E_i$}.  Hence, \mbox{$\langle\mkern1mu E\mkern2mu\vert\mkern12mu\rangle = \bigast\limits_{i\in I} \langle E_i \rangle \le
\bigast\limits_{i\in I} H_i = \langle\mkern1mu E\mkern2mu\vert\mkern12mu\rangle.$}
It follows that \mbox{$\langle E_i \rangle = H_i$} and, hence,
 \mbox{$E_i$} is a basis of~\mbox{$H_i$}.
\end{proof}

\begin{cvl}\label{lem:cvl}  With
 Hypotheses~{\normalfont\ref{hyps:1}}, suppose that
\mbox{$\Wh(Z \rel E_Z)$} has no Whitehead cut-ver\-tices.
If $Z$~is a  sub-basis  of
 \mbox{$\langle\mkern1mu E\mkern2mu\vert\mkern12mu\rangle$},
then~\mbox{$Z \subseteq E^{\pm 1}$}.  Hence,
  $Z$~is a  sub-basis  of
 \mbox{$\langle\mkern1mu E\mkern2mu\vert\mkern12mu\rangle$} if and only if
 \mbox{$Z {\,\cap\,} Z^{-1} \hskip-.5pt = \hskip-.5pt \emptyset$}
and \mbox{$Z  \subseteq E^{\pm 1}$};
in this event,
 \mbox{$Z \cup (E{-} Z^{\pm1})$} is a
basis of~\vspace{2mm}\mbox{$\langle\mkern1mu E\mkern2mu\vert\mkern12mu\rangle$}.
\end{cvl}

\begin{proof}  Let $E'$ be a basis of \mbox{$\langle\mkern1mu E\mkern2mu\vert\mkern12mu\rangle$}
that contains $Z$.   A classic $E'$-length argument  due to  Nielsen shows that
\mbox{$E' \cap\mkern2mu \langle E_Z  \rangle$} is contained in some basis   \mbox{$X$}
of~\mbox{$\langle E_Z  \rangle$}; we shall mention Schreier's proof in
 Review~\ref{rev:sch}.
Now   \mbox{$\langle  E_{Z}\rangle
=  \bigast\limits_{x\in X}   \langle x \rangle$}
and \mbox{$Z \subseteq E' \cap \langle E_Z  \rangle  \subseteq
X \subseteq \bigcupp\limits_{x \in  X}
  \langle x \rangle$}. By Corollary~\ref{cor:sto},  \mbox{$E_Z$}
contains a basis of each~\mbox{$\langle x \rangle$}, necessarily $\{x\}$ or~$\{\overline x\}$.
Thus \mbox{$E_Z^{\pm 1} \supseteq     X \supseteq Z$}.
\end{proof}

We shall use the following strong form    in the next section.

\begin{corollary}\label{cor:W2B} If
  $Z$ is a  sub-basis  of~\mbox{$\langle\mkern1mu E\mkern2mu\vert\mkern12mu\rangle$}
and~\mbox{$Z \not \subseteq E^{\pm 1}$}, then there exists some
 \mbox{$\mathbf{C} \,{\in}\, \cuts(E)$} such that \mbox{$\Wh(Z \rel E) \subseteq \Wh(\,\mathbf{C} )$}
and
\mbox{$\vert \vert Z^{\mkern2mu\overline \phi_{\mathbf{C} }} \vert \vert_E  < \vert \vert Z  \vert \vert_E$}.
\end{corollary}

\begin{proof}
By the contrapositive of Lemma\,\ref{lem:cvl}, \mbox{$\Wh(Z \rel E_Z)$}
has a Whitehead cut\d1ver\-tex.  The result now follows from Algorithm\,\ref{alg:preW}.
\end{proof}

It remains to discuss free-product factors.

\begin{review}\label{rev:1b} $\bullet$  We now sketch a   proof of a result of Kurosh:
 for any subgroups $H$~and $K$ of
any group~$F$,
if $K$ is a free-product factor of~$F$, say \mbox{$F =   K {\ast}L$},
then \mbox{$H \cap K$} is a free-product factor of $H$.

We shall use Bass-Serre theory, although for our purposes
the case  \mbox{$F\hskip-1.1pt = \hskip-1.1pt \langle\mkern1mu E\mkern2mu\vert\mkern12mu\rangle$} and the
graph-theoretic techniques of John\,R.\,\,Stallings~\cite{S83} would  suffice.
  
We may view  \mbox{$\BST(F, (K,L))$} as an $H$-tree,
and then the vertex $1K$ can be extended to a fundamental $H$-transversal.
The resulting graph of groups has \mbox{$H \cap K$} as one of the vertex-groups
and all the edge-groups are trivial.  By another result of Bass and Serre,
\mbox{$H \cap K$} is a free-product factor of $H$.  See,  for example, \cite[Theorem~I.4.1]{DD}.

It follows that, for any group,
  the set of all its free-product factors    is closed under finite intersections.

$\bullet$  Recall   Hypotheses~{\normalfont\ref{hyps:1}} and set
 \mbox{$F \coloneq \langle\mkern1mu E\mkern2mu\vert\mkern12mu\rangle$}.
Now $\vert E \vert$  bounds the length of any strictly descending chain of free-product factors of $F$.  Hence,
the set of all the free-product factors of $F$ is closed under arbitrary intersections.

In particular, \mbox{$\Cl(Z)$}, the intersection of all the free-product factors of $F$ containing~$Z$,
 is the $\subseteq$-smallest free-product factor of $F$ containing~$Z$.

By Kurosh's result again,
 \mbox{$\Cl(Z) \cap \langle E_Z \rangle$} is
 a free-product factor of
\mbox{$\langle E_Z \rangle$}.  However, \mbox{$\langle E_Z \rangle$} contains \mbox{$\Cl(Z)$},
since   \mbox{$\langle E_Z \rangle$} is a free-product factor of
\mbox{$F$} which contains~$Z$.  Thus,
 \mbox{$\Cl(Z)$} is a free-product factor of~\mbox{$\langle E_Z \rangle$}.
In particular, the bases of \mbox{$\Cl(Z)$} are the minimal-size
supports of $Z$ with respect to bases of~$F$.  
 \end{review}

\begin{stong} \label{lem:sto} \hskip4pt With \hskip4pt
 Hypotheses \hskip4pt {\normalfont\ref{hyps:1}}, \hskip4pt suppose \hskip4pt that\linebreak
\mbox{$\Wh(Z \rel E_Z)$} has no Whitehead cut-ver\-tices.
 Then  \mbox{$E_Z$} is a basis of \mbox{$\Cl(Z)$}, and, for each
 free-product factorization\vspace{-2.5mm}
\mbox{$\Cl(Z) = \bigast\limits_{i\in I} H_i$} such that
\mbox{$Z \subseteq   \bigcupp\limits_{i\in I} H_i$}, the set  
\mbox{$E_Z$}  contains a basis of each~\mbox{$H_i$}.
\end{stong}

\begin{proof}
We saw in Review~\ref{rev:1b} that there exists some free-product factorization
\mbox{$\langle E_Z \rangle =  \Cl(Z) {\ast} K$}, and it is clear that
\mbox{$Z \subseteq  \Cl(Z) \cup K$}.
By Corollary~\ref{cor:sto},   \mbox{$E_Z$} contains some basis \mbox{$E'$} of~\mbox{$\Cl(Z)$}.
Since \mbox{$Z \subseteq \Cl(Z) = \langle E'\rangle$}, we see that
\mbox{$\supp(Z \rel E) \subseteq  E'$}, that is,
\mbox{$E_Z  \subseteq  E'$}.  Hence,
  \mbox{$E_Z$} is a basis of \mbox{$\Cl(Z)$}.  The result   now follows from Corollary~\ref{cor:sto}.
\end{proof}

\section{A strengthened Clifford-Goldstein  algorithm}\label{sec:CG}

Clifford and Goldstein~\cite{CG0}
produced an ingenious algorithm which  returns  an element of  \mbox{$\langle Z \rangle$} that lies in a
 basis of~\mbox{$\langle\mkern1mu E\mkern2mu\vert\mkern12mu\rangle$}
or reports that no
element of  \mbox{$\langle Z \rangle$} lies in a basis
of~\mbox{$\langle\mkern1mu E\mkern2mu\vert\mkern12mu\rangle$}. They   used Whitehead's  three-manifold
techniques to
  construct a sufficiently large finite
set of finitely generated subgroups of \mbox{$\langle\mkern1mu E\mkern2mu\vert\mkern12mu\rangle$}
whose elements of sufficiently bounded $E$-length give the desired information.

In this section, we restructure their argument,    bypassing the topology and obtaining  a
 less complicated, more powerful
algorithm which
 yields as output a basis $E''$  of  \mbox{$\langle\mkern1mu E\mkern2mu\vert\mkern12mu\rangle$}
which maximizes   \mbox{$\vert E'' \cap \langle Z \rangle\vert$}.
 In particular,   \mbox{$E'' \cap \langle Z\rangle = \emptyset$}
 if and only if no element of  \mbox{$\langle Z \rangle$}
lies in a basis of \mbox{$\langle\mkern1mu E\mkern2mu\vert\mkern12mu\rangle$}.
We construct  a smaller sufficiently large finite
set of finitely generated subgroups of \mbox{$\langle\mkern1mu E\mkern2mu\vert\mkern12mu\rangle$}
whose intersections with $E$ give  the desired information.

\medskip

To fix notation, we sketch the proof of Schreier~\cite[publ.\,\,1927]{Sch}
that  subgroups of free groups are free. The
finitely generated case  had been proved by J.\,Nielsen~\mbox{\cite[publ.\,1921, in Danish]{N}}.

\begin{review}\label{rev:sch}  With  Hypotheses~{\normalfont\ref{hyps:1}},
 set \mbox{$F \coloneq \langle\mkern1mu E\mkern2mu\vert\mkern12mu\rangle$} and
 \mbox{$T \coloneq \Cayley(F,E)$};
see Review\,\ref{rev:trees}.
Let $H$ be a subgroup of $F$.
The vertices of the \textit{Schreier graph}
  \mbox{$H\backslash T$} are the cosets  \mbox{$Hg$},  \mbox{$g \in F$},
  the   basepoint is~\mbox{$H1$}, and we write
  \mbox{$\operatorname{edge}(v \xrightarrow{\bullet (e)}  ve) \coloneq (v,e) \in (H\backslash F)\,{\times}\,E$}.
The  graph \mbox{$H\backslash T$}  is  connected.
Let  \mbox{$\pi(H\backslash T, H1)$} denote the
 fundamental group of~\mbox{$H\backslash T$} at the basepoint~\mbox{$H1$}.
Each   (reduced)  \mbox{$H\backslash T$}-path   from  \mbox{$H1$}
to itself will be viewed as a (reduced)   \mbox{$E^{\pm 1}$}-ex\-pres\-sion  for some element of $H$;
for example, we would view

\centerline{\mbox{$(H1 \xrightarrow{ \bullet (e_1)} He_1
\xleftarrow{\bullet (e_2)} He_1\overline e_2   \xrightarrow{ \bullet(e_3)} He_1\overline e_2 e_3 =H1)$}}

\noindent
as the \mbox{$E^{\pm 1}$}-expression \mbox{$e_1 \overline e_2 e_3$} for an
element of  $H$.
Hence, we may identify  \mbox{$\pi(H\backslash T, H1)$} with~$H$.

Choose a maximal subtree  \mbox{$Y'$} of  \mbox{$H\backslash T$}
and let   \mbox{$Y''$} denote the complement of~\mbox{$Y'$} in  \mbox{$H\backslash T$};
then  \mbox{$Y''$} is a set of edges.  Each element $y''$ of  \mbox{$Y''$} determines
the element of  \mbox{$\pi(H\backslash T, H1)$}  that travels in
 \mbox{$Y'$} from  \mbox{$H1$} to the initial vertex of $y''$, travels along $y''$, and then
travels in  \mbox{$Y'$} from the terminal vertex of $y''$ to  \mbox{$H1$}.
By letting $y''$ range over  \mbox{$Y''$}, we get a subset $S$ of
  \mbox{$\pi(H\backslash T, H1)$}.  By collapsing  the tree \mbox{$Y'$} to a vertex, we find that
$S$ freely generates
  \mbox{$\pi(H\backslash T, H1)\,\,(=H)$}.

The vertices and edges involved in $S$ form a connected basepointed subgraph of
 \mbox{$H\backslash T$} denoted  \mbox{$\core(H \rel E)$}.   An alternative description is that
 \mbox{$\core(H \rel E)$} consists of those vertices and edges
that are involved in the reduced \mbox{$H\backslash T$}-paths     from  \mbox{$H1$}
to itself.
Thus, \mbox{$\pi(\core(H \rel E), H1) =H$} and
 \mbox{$\core(H \rel E)$} is the $\subseteq$-smallest subgraph \vspace{-.5mm} of
 \mbox{$H\backslash T$} with this property.

For each  \mbox{$h \in E \cap H$}, it is clear that \mbox{$\operatorname{edge}(H1 \xrightarrow{\bullet(h)} Hh = H1)$}
is not in the tree~\mbox{$Y'$}, and, hence,   \mbox{$h \in S$}.
Thus, \mbox{$E \cap H \subseteq S$}.   (I am indebted to Clifford and Goldstein
for this  paragraph.)
\end{review}

\begin{algorithm}\label{alg:st}    Stallings' \textit{core algorithm}~\cite[Algorithm\,5.4]{S83}
has the following structure.

 With  Hypotheses~{\normalfont\ref{hyps:1}},  we
shall suppress the information that the vertices of  \mbox{$\core(\langle Z \rangle \rel E)$}
are certain cosets, and we shall
build a basepointed $E$-labelled graph, denoted  \mbox{$\mcore(\langle Z \rangle \rel E)$}, that
has an abstract  set as vertex-set and is
 isomorphic to \mbox{$\core(\langle Z \rangle \rel E)$} as basepointed $E$-labelled graph.

For each  \mbox{$z \in Z$}, we easily build   \mbox{$\mcore(\langle z \rangle \rel E)$} as a
basepointed $E$-labelled lollipop graph, possibly trivial,
using the reduced  \mbox{$E^{\pm 1}$}-expression for~$z$.

We next amalgamate all these   lollipop graphs at their basepoints.
Throughout the construction, each edge will  be assigned an expression of the form
  \mbox{$\operatorname{edge}(v \xrightarrow{\bullet (e)} w)$} with $v$, $w$ vertices and  \mbox{$e \in E$}, but, for the moment,
 the expression need not determine the edge.  While possible,
we identify some  distinct pair of  edges  with expressions
\mbox{$\operatorname{edge}(v \xrightarrow{\bullet (e)} w)$} and
 \mbox{$\operatorname{edge}(v' \xrightarrow{\bullet (e)} w')$}
where  \mbox{$v=v'$} or  \mbox{$w=w'$} or both;
identifying the edges entails identifying $w$ with $w'$ or  $v$ with $v'$   or neither,
respectively.
 When no such pair of distinct edges is left, the
 procedure has yielded a basepointed $E$-labelled graph isomorphic to
 \mbox{$\core(\langle Z \rangle \rel E)$};
here,  expressions   \mbox{$\operatorname{edge}(v \xrightarrow{\bullet (e)} w)$}  do determine  edges. \qed
\end{algorithm}

 Stallings gave the name \textit{folding} to the foregoing edge-identifying process.
The process itself had long been used  unnamed, notably by Lyndon  in his work on
planar diagrams, where each nontrivial lollipop graph has a two-cell attached making a contractible CW-complex.

\smallskip

We now give the (strange) key construction of~\cite[Theorem~1]{CG0}.

\begin{notation}\label{not:partial}  With  Hypotheses~{\normalfont\ref{hyps:1}}, fix
 \mbox{$\mathbf{C} = (\,\null_0D, \null_1D, e_\star\,) \in \cuts(E)$}, and set
 \mbox{$F \coloneq \langle\mkern1mu E\mkern2mu\vert\mkern12mu\rangle$},
 \mbox{$\eta \coloneq\eta_{\,\mathbf{C}}$}, \mbox{$d_\star\, \coloneq e_\star^{2\eta{-}1}$},
and \mbox{$\phi \coloneq \phi_{\,\mathbf{C}}$}; see  Notation\,\ref{not:Ds}.

We first construct an $F$-map  \mbox{$\psi_{\mathbf{C}}$} from the edge-set \vspace{-.5mm} of
\mbox{$T \coloneq \Cayley(F,E)$} to the edge-set of    \mbox{$T' \coloneq \Cayley(F, E^{\,\phi})$}.
For any  \mbox{$\operatorname{edge}(g \xrightarrow{\bullet (e)}ge)  \in  F\,{\times}\,E$},
there exists a unique   \mbox{$(\alpha,\beta) \in \{0,1\}^{\times 2}$}
such that \mbox{$e  \in  \null_\alpha E_\beta$} and
\mbox{$e^\phi = d_\star^\alpha e  \overline d_\star^{\,\beta}$};
if \mbox{$e^{\pm 1} \ne e_{\star\,}^{\pm 1}$}, these two conditions are equivalent, while
if \mbox{$e^{\pm 1} = e_{\star\,}^{\pm 1}$}, the two conditions together say that  \mbox{$\alpha = \beta = \eta$}.
We set  \mbox{$(\operatorname{edge}(g \xrightarrow{\bullet (e)}  ge))^{\psi_{\mathbf{C}}}
\coloneq  \operatorname{edge}(g \overline d_\star^{\mkern3mu\alpha} \xrightarrow{\bullet(e^\phi)} ge \overline d_\star^{\mkern3mu\beta})$};
we emphasize that  no action of  \mbox{$\psi_{\mathbf{C}}$}  on vertices is being defined.
It is clear that  \mbox{$\psi_{\mathbf{C}}$} is an $F$-map.

Let   \mbox{$H$} be a finitely generated subgroup of $F$.
Then  \mbox{$\psi_{\mathbf{C}}$} induces a set map  from the edge-set of
 \mbox{$ H \backslash T$} to the edge-set
of  \mbox{$H \backslash T'$}, and the image of the edge-set of
\mbox{$\core( H  \rel E)$} under this induced map is
 then the edge-set of a unique subgraph  $X$  of
 \mbox{$H \backslash T'$} with the full vertex-set,  \mbox{$H \backslash F$}.
  Let \mbox{$K \coloneq \pi(X, H1) \le \pi(H \backslash T', H1)$}.
We may identify the latter group with $H$,  where \mbox{$(H \backslash T')$}-paths are
 \mbox{$(E^{\,\phi})^{\pm 1}$}-ex\-pressions.
We set \mbox{$\partial_{\,\mathbf{C}\mkern-1mu} H \coloneq K^{\overline \phi} \le H^{\overline \phi}$}.
Recall that \mbox{$\mcore(H \rel E)$}
was constructed in Algorithm\,\ref{alg:st};
we shall be viewing \mbox{$\partial_{\,\mathbf{C}\mkern-1mu}$} as a graph
operation that converts   \mbox{$\mcore(H \hskip-2pt \rel\hskip-2pt E)$} into
 \mbox{$\mcore(\partial_{\,\mathbf{C}\mkern-1mu} H  \hskip-2pt\rel\hskip-2pt E)$}.
\end{notation}

\begin{lemma}\label{lem:CG} With the foregoing notation,  the following hold for
 \mbox{$\partial_{\,\mathbf{C}} H \le H^{\overline \phi_{\mathbf{C}}}$}.\vspace{-3mm}
\begin{enumerate}[{\rm (i)}]
\setlength\itemsep{-5pt}
\item
 \mbox{$\mcore( \partial_{\,\mathbf{C} } H \rel E)$} may be  constructed algorithmically.
\item    \mbox{$\core(H \rel E)$} has at least as many edges as
 \mbox{$\core(\partial_{\,\mathbf{C}\mkern-1mu} H \rel E)$}.
\item For each     \mbox{$z \in H$}, if   \mbox{$\Wh(\{z\} \rel E) \subseteq \Wh(\,\mathbf{C})$}, then
 \mbox{$z^{\,\overline \phi_{\mathbf{C}'}} \in \partial_{\,\mathbf{C} } H$}.
\item  If $Y$ is any  sub-basis  of
 \mbox{$\langle\mkern1mu E\mkern2mu\vert\mkern12mu\rangle$}  such that  \mbox{$Y  \subseteq H$} and  \mbox{$Y \not \subseteq E^{\pm 1}$}, then
there exists some  \mbox{$\mathbf{C}' \in \cuts(E)$} such that
\mbox{$Y^{\,\overline \phi_{\mathbf{C}'}}\subseteq \partial_{\,\mathbf{C}'\mkern-1mu} H$} and
\mbox{$\vert \vert Y^{\,\overline \phi_{\mathbf{C}'}}\vert \vert_E < \vert \vert Y \vert \vert_E$}.
\end{enumerate}  \vspace{-2mm}
\end{lemma}

\begin{proof}  (i).   Since \mbox{$K^{\overline \phi}  =  \partial_{\,\mathbf{C}\mkern-1mu} H $},
there is a natural graph isomorphism that maps   \mbox{$\core(K \rel E^{\,\phi})$}
to \mbox{$\core( \partial_{\,\mathbf{C}\mkern-1mu} H \rel E)$}, changing each
\mbox{$Kg \xrightarrow{\bullet(e^\phi)} Kg(e^\phi)$} to
\mbox{$K^{\overline \phi} g^{\overline \phi} \xrightarrow{\bullet(e)} K^{\overline \phi}g^{\overline \phi}e$}.
Hence, there is a natural graph isomorphism that maps   \mbox{$\mcore(K \rel E^{\,\phi})$}
to \mbox{$\mcore( \partial_{\,\mathbf{C}\mkern-1mu} H \rel E)$}, changing each
\mbox{$v \xrightarrow{\bullet(e^\phi)} w$} to
\mbox{$v \xrightarrow{\bullet(e)}   w$}; the labels on the non-basepoint vertices are irrelevant.
Thus, it suffices to  algorithmically construct  \mbox{$\mcore(K \rel E^{\,\phi})$} from  \mbox{$\mcore(H \rel E)$}.

If \mbox{$d_\star \in E$}, resp. \mbox{$\overline d_\star \in E$},
we say that a vertex\vspace{-1mm} $v$ of \mbox{$\mcore(H \rel E)$} \textit{has a neighbour 
\mbox{$v \overline d_\star$}} if an  edge of the form 
  \mbox{$\operatorname{edge}(w \xrightarrow{\bullet(d_\star)} v$)},
resp. \mbox{$\operatorname{edge}(v \xrightarrow{\bullet(\overline d_\star)} w$)},
lies in \mbox{$\mcore(H \rel E)$}; in this event, we say that \mbox{$w$} is
\mbox{$v \overline d_\star$}.
We simultaneously add to \mbox{$\mcore(H \rel E)$},
for every vertex $v$ that does not have a neighbour~\mbox{$v \overline d_\star$},
  a valence\d1zero vertex with  label  \mbox{$v \overline d_\star$}.

Next, in \mbox{$\mcore(H \rel E)$} adorned with the valence-zero vertices,
we   simultaneously  replace each    \mbox{$\operatorname{edge}(v \xrightarrow{\bullet(e)} w$)}
  with
 \mbox{$\operatorname{edge}(v\overline d_\star^{\,\alpha} \xrightarrow{\bullet(e^\phi)}
w\overline d_\star^{\,\beta})$} \vspace{.5mm}
for the unique  \mbox{$(\alpha,\beta) \in \{0,1\}^{\times 2}$}
such that \mbox{$e  \in  \null_\alpha  E_\beta$} and
\mbox{$e^\phi = d_\star^\alpha e \overline  d_\star^{\,\beta}$}.
This particular operation alters incidence maps and edge labellings,
but not the vertex-set or the edge-set.

 In the resulting finite  graph, we then keep only the component
that has the basepoint.  We next
 successively delete non-basepoint, valence-one vertices and their (unique) incident edges, while possible.  When this is
no longer possible, we have  constructed   \mbox{$\mcore(K \rel E^{\,\phi})$}   algorithmically.

 (ii).  It is clear from the constructions that  \mbox{$\core(H \rel E)$} has at least as many
 edges as  \mbox{$\core(K \rel E^{\,\phi})$}, which in turn has the same number of
edges as  \mbox{$\core(\partial_{\,\mathbf{C}\mkern-1mu} H \rel E)$}.

\hskip-4pt(iii).  \hskip-4pt Consider any expression
\mbox{$ Hg \xrightarrow{\bullet (e)} Hge$}\vspace{1mm}
corresponding to an edge or inverse edge in \mbox{$\core(H \rel E)$},
  and consider any \mbox{$(\alpha,\beta) \in \{0,1\}^{\times 2}$} such that
\mbox{$e \in   \null_\alpha D_\beta$}.\vspace{1mm}  Then
\mbox{$d_\star^{\,\alpha} e \overline d_\star^{\,\beta} \in \{e^\phi, 1\}$},\vspace{.5mm} for,
if \mbox{$d_\star^{\,\alpha} e \overline d_\star^{\,\beta} \ne e^\phi$},
then either \mbox{$e = e_\star\,$}, \mbox{$\alpha = 1{-}\eta$}, \mbox{$\beta = \eta$},
\mbox{$ d_\star^{\,\alpha} e \overline d_\star^{\,\beta} = d_\star^{1-\eta-\eta} e_\star\, = 1$},
or \mbox{$e = \overline e_\star\,$}, \mbox{$\alpha = \eta$}, \mbox{$\beta = 1{-}\eta$},
\mbox{$d_\star^{\,\alpha} e   \overline  d_\star^{\,\beta} = d_\star^{\eta-1+\eta} \overline e_\star\, = 1$}.
This means that the expression \mbox{$Hg \overline d_\star^{\,\alpha}
\xrightarrow{\bullet(d_\star^\alpha e \overline d_\star^{\,\beta})} Hge \overline d_\star^{\,\beta}$}
corresponds to an edge, inverse edge, or   equality  in the graph~$X$ of Notation\,\ref{not:partial}.

Suppose that  \mbox{$z \in H$} and let  \mbox{$e_1 e_2 \cdots e_n$} represent the
reduced  \mbox{$E^{\pm 1}$}-expression for $z$.
We then have a corresponding reduced \mbox{$H \backslash T$}-path   from   \mbox{$H1$}
 to itself, which we may write  in  \mbox{$H \backslash \Cayley(F, E^{\pm 1})$} as

\centerline {\mbox{$H1 \xrightarrow{\bullet(e_1)} He_1
\xrightarrow{\bullet(e_2)}  He_1e_2 \xrightarrow{\bullet(e_3)} \cdots
\xrightarrow{\bullet(e_n)}  He_1e_2 \cdots e_n = Hz = H1.$}}

\noindent The  \mbox{$H \backslash T$}-path  must then stay within the subgraph  \mbox{$\core(H \rel E)$}.

Suppose further that \mbox{$\Wh(\{z\} \rel E) \subseteq \Wh(\mathbf{C})$}.  This means that
there exists a (unique) set map \mbox{$\{0, 1, \ldots, n\} \to \{0,1\}$}, \mbox{$i \mapsto \chi_i$},
such that \mbox{$e_i \in \null_{\chi_{i-1}} D_{\chi_i}$},   \mbox{$i  = 1,2,\ldots, n$}, 
and \mbox{$\chi_0 = \chi_{n} = 0$}.   In our \mbox{$\core(H \rel E)$}-path,
let us change each vertex \mbox{$ H e_1 \cdots  e_{i}$} to \mbox{$H e_1 \cdots  e_{i } \overline d_\star^{\,\chi_{i}}$}
and  each step
\mbox{$H e_1 \cdots e_{i-1} \xrightarrow{\bullet(e_{i})} H e_1 \cdots e_{i-1}e_{i}$} to
\mbox{$H e_1 \cdots e_{i-1} \overline d_\star^{\,\chi_{i-1}}
\xrightarrow{\bullet( d_\star^{\,\chi_{i-1}} e_{i}\overline d_\star^{\,\chi_{i}}) }
H e_1 \cdots e_{i-1} e_{i} \overline d_\star^{\,\chi_{i}},$}
which we have seen corresponds to an edge, inverse edge, or  equality in $X$.
We thus obtain an $X$-path    from    \mbox{$H1$} to itself that reads
an  \mbox{$((E^{\,\phi})^{\pm 1}\cup\{1\})$}-expression for $z$.
This shows that \mbox{$z \in \pi(X, H1) = K$}, as desired.

(iv).   By Corollary\,\ref{cor:W2B},  there exists
 \mbox{$\mathbf{C}'   \in \cuts(E)$}
such that \mbox{$\vert \vert Y^{\mkern2mu\overline \phi_{\mathbf{C}'}}
\vert \vert_{E} \hskip-.9pt<\hskip-.9pt \vert \vert Y  \vert \vert_{E}$} and
\mbox{$\Wh(Y \rel E)\subseteq  \Wh(\,\mathbf{C}')$}.
By (iii), \mbox{$Y^{\mkern2mu\overline \phi_{\mathbf{C}'}}\subseteq \partial_{\,\mathbf{C}'\mkern-1mu} H$}.
\end{proof}

We now give a construction that is a somewhat less complicated variant of
 the algorithm of Clifford and Goldstein~\cite{CG0}.

\begin{notation}  \label{not:CG} With  Hypotheses~{\normalfont\ref{hyps:1}}, let
  \mbox{$\mathcal{F}$} denote the set of all finitely generated subgroups of
 \mbox{$\langle\mkern1mu E\mkern2mu\vert\mkern12mu\rangle$}.  Let
 \mbox{$\Gamma$}~denote the graph whose vertex-set is
 \mbox{$\mathcal{F}$} and whose edge-set is  \mbox{$\mathcal{F}\,{\times}\, \cuts(E)$}
where  each edge  \mbox{$(H, \mathbf{C}) \in \mathcal{F}\,{\times}\, \cuts(E)$} has initial vertex
$H$ and terminal vertex  \mbox{$\partial_{\,\mathbf{C}\mkern-1mu} H$}; see Notation~$\ref{not:partial}$.

Set  \mbox{$G \coloneq \langle Z \rangle \in \mathcal{F}$}.
 Let  \mbox{$(G{\blacktriangleleft})$}  denote
the subgraph of  \mbox{$\Gamma$}
that radiates out from~$G$, that is,   \mbox{$(G{\blacktriangleleft})$} is the smallest subgraph of
 \mbox{$\Gamma$} that has~$G$ as a vertex and is closed in  \mbox{$\Gamma$} under
the operation of adding to each vertex $H$ each outgoing edge
 \mbox{$(H, \mathbf{C})$} and its terminal vertex  \mbox{$\partial_{\,\mathbf{C}} H$}.

For each   \mbox{$n \ge 0$}, each element  \mbox{$(\mathbf{C}_i)_{i=1}^n$} of  $(\cuts(E))^{\times n}$
determines the oriented  \mbox{$(G{\blacktriangleleft})$}-path with  the
edge-sequence  \mbox{$(H_{i}\xrightarrow{(H_{i}, \mathbf{C}_i)}H_{i+1})_{i=1}^n$}
where  \mbox{$H_1 = G$} and  \mbox{$H_{i+1} = \partial_{\mathbf{C}_i} H_{i}$} for
 \mbox{$i = 1,\ldots, n$}.
To simplify notation, we shall say that  \mbox{$(\mathbf{C}_i)_{i=1}^n$} itself
  is an oriented \mbox{$(G{\blacktriangleleft})$}-path  with initial vertex $G$.

We usually think of a vertex $H$ of   \mbox{$(G{\blacktriangleleft})$} 
as the graph  \mbox{$\mcore(H \rel E)$}, for ease of recognition.
  We shall see that we are interested in finding a vertex 
that maximizes the number of  loops at the basepoint.
\end{notation}

\begin{theorem}\label{thm:CG} With the foregoing notation, the following hold.
\vspace{-2mm}
\begin{enumerate}[{\rm (i)}]
\setlength\itemsep{-5pt}
\item  \mbox{$(G{\blacktriangleleft})$} is an algorithmically constructible
 finite graph whose vertices are viewed as
finite, $E$-labelled, basepointed graphs.

\item For each vertex $H$ of    \mbox{$(G{\blacktriangleleft})$}, there is an
algorithmically constructible  oriented \mbox{$(G{\blacktriangleleft})$}-path   \mbox{$(\mathbf{C}_i)_{i=1}^n$}
  from $G$ to $H$, \mbox{$H = \partial_{\mathbf{C}_n} \cdots 
 \partial_{\mathbf{C}_1} G \le G^{\,\overline \phi_{\mathbf{C}_1} 
 \cdots  \overline  \phi_{\mathbf{C}_n}}$}, and
\mbox{$(E \cap H)^{\phi_{\mathbf{C}_n} \cdots \phi_{\mathbf{C}_1}} \subseteq  E'' \cap\, G$},
where  \mbox{$E'' \coloneq E^{\mkern2mu\phi_{\mathbf{C}_n} \cdots 
\phi_{\mathbf{C}_1}}$}.

\item  For each basis $E''$   of
  \mbox{$\langle\mkern1mu E\mkern2mu\vert\mkern12mu\rangle$}, there exists
 some vertex $H$ of  \mbox{$(G{\blacktriangleleft})$} such that
 \mbox{$\vert E \cap H \vert \ge \vert E'' \cap \,G \vert$}.
\end{enumerate}  \vspace{-2mm}
\end{theorem}

\begin{proof}  (i).  For each \mbox{$H \in \mathcal{F}$}, 
if $n$ denotes the number of edges in \mbox{$\core (H \rel E)$}, it is clear from Review~\ref{rev:sch} that
$H$ can be generated by $n$-or-less elements of  
  \mbox{$\langle\mkern1mu E\mkern2mu\vert\mkern12mu\rangle$} of $E$-length $2n$-or-less.
By Lemma\,\ref{lem:CG}(ii),  \mbox{$(G{\blacktriangleleft})$} is finite.
By Lemma\,\ref{lem:CG}(i), we may
use a depth-first search to construct a maximal subtree of  \mbox{$(G{\blacktriangleleft})$}.
We then add  the missing edges of  \mbox{$(G{\blacktriangleleft})$},
although this is optional for our purposes.

(ii) is clear.

(iii). It follows from Lemma\,\ref{lem:CG}(iv) that  there exists
some   \mbox{$(\mathbf{C}_i)_{i=1}^n$}  such that
 \mbox{$(E'' \cap\, G)^{\,\overline \phi_{\mathbf{C}_1}\overline \phi_{\mathbf{C}_2}
 \cdots  \overline  \phi_{\mathbf{C}_n}} \subseteq   E^{\pm 1} \cap \,\partial_{\mathbf{C}_n}
 \cdots \partial_{\mathbf{C}_2}
 \partial_{\mathbf{C}_1} G $}.
\end{proof}

We now construct a basis \mbox{$E''$}  of  \mbox{$\langle\mkern1mu E\mkern2mu\vert\mkern12mu\rangle$}
which maximizes   \mbox{$\vert E'' \cap \langle Z \rangle\vert$}.

\begin{algorithm}\label{alg:CG}  Recall   Hypotheses~{\normalfont\ref{hyps:1}}.

$\bullet$ Set \mbox{$G \coloneq \langle Z \rangle$} and construct \mbox{$\mcore(G \rel E)$}; see Algorithm\,\ref{alg:st}.

$\bullet$ Construct~\mbox{$(G{\blacktriangleleft})$} from  \mbox{$\mcore(G \rel E)$}; see Theorem\,\ref{thm:CG}(i).

$\bullet$ In \mbox{$(G{\blacktriangleleft})$}, find a vertex $H$    maximizing
  the number of loops at the basepoint of \mbox{$\mcore(H \rel E)$}, that is,
maximizing \mbox{$\vert E \cap H\vert$}.

$\bullet$ Find an  oriented \mbox{$(G{\blacktriangleleft})$}-path \mbox{$(\mathbf{C}_i)_{i=1}^n$}   from $G$ to~$H$;
 see Theorem\,\ref{thm:CG}(ii).

$\bullet$ Return \mbox{$E'' \coloneq E^{\mkern2mu\phi_{\mathbf{C}_n} \cdots \phi_{\mathbf{C}_2}  \phi_{\mathbf{C}_1}}$},
 a basis  of  \mbox{$\langle\mkern1mu E\mkern2mu\vert\mkern12mu\rangle$}
which maximizes  \mbox{$\vert E'' \cap \langle Z \rangle\vert$} by Theorem\,\ref{thm:CG}(ii),(iii). \qed
\end{algorithm}

\bibliographystyle{plain}

\noindent
\textsc{Departament de  Matem\`atiques,  \\ Universitat Aut\`onoma de Barcelona, \\
08193 Bellaterra (Barcelona), Spain} \\  
\noindent \emph{email}{:\;\;}\url{dicks@mat.uab.cat}  \qquad \emph{URL}{:\;\;}\url{http://mat.uab.cat/~dicks/}
\end{document}